\documentclass[11pt,twoside]{amsart}
\usepackage{eurosym}
\usepackage{amscd,amsfonts,amsmath,amssymb,amsthm,latexsym}
\usepackage{enumerate,array,srcltx,amsmath}
\usepackage{leqno}
\usepackage[dvips]{graphicx}
\usepackage{multicol}
\usepackage{a4}
\usepackage{fancyhdr}
\usepackage{caption}
\usepackage{subcaption}
\usepackage{mathrsfs}
\usepackage[driverfallback=pdfmx,,colorlinks,bookmarksopen,bookmarksnumbered,citecolor=blue, urlcolor=blue]{hyperref}
\usepackage{color}
\newcommand\blfootnote[1]{%
  \begingroup
  \renewcommand\thefootnote{}\footnote{#1}%
  \addtocounter{footnote}{-1}%
  \endgroup
}
\definecolor{darkblue}{rgb}{0,0,0.545098}
\definecolor{darkgreen}{rgb}{0,0.392157,0}
\newtheorem{theorem}{Theorem}[section]
\newtheorem{lemma}[theorem]{Lemma}
\newtheorem{proposition}[theorem]{Proposition}
\theoremstyle{definition}
\newtheorem{definition}[theorem]{Definition}

\theoremstyle{remark}
\newtheorem{remark}[theorem]{Remark}
\newtheorem{remarks}[theorem]{Remarks}
\numberwithin{equation}{section}

\makeatother

\begin{document}
\title[Lower Bound and optimality for a nonlinearly damped Timoshenko system]%
{Lower Bound and optimality for a nonlinearly damped Timoshenko system with
thermoelasticity }
\author[Bchatnia ]{Ahmed Bchatnia$^1$}
\address{$^1$ UR ANALYSE NON-LIN\'EAIRE ET G\'EOMETRIE, UR13ES32, Department of
Mathematics, Faculty of Sciences of Tunis, University of Tunis El-Manar,
2092 El Manar II, Tunisia}
\email{ahmed.bchatnia@fst.rnu.tn}
\author[Chebbi]{Sabrine chebbi$^1$}
\email{sabrinech.chebbi91@gmail.com}
\author[Hamouda]{Makram Hamouda$^{ 2^{\ast}}$}
\address{$^2$ Institute for Scientific Computing and Applied Mathematics, Indiana
University, 831 E. 3rd St., Rawles Hall, Bloomington IN 47405, United States}
\email{mahamoud@indiana.edu}
\author[Soufyane]{Abdelaziz Soufyane$^3$}
\address{$^3$ Department of Mathematics, College of Sciences, University of
Sharjah, P.O.Box 27272, Sharjah, UAE.}
\email{asoufyane@sharjah.ac.ae}

\maketitle

\begin{abstract}
In this paper, we consider a vibrating nonlinear Timoshenko system with
thermoelasticity with second sound. We first investigate the 
stability of this system, then we devote our efforts to obtain the
strong lower energy estimates using  Alabau-Boussouira's energy comparison principle introduced in \cite{2} (see also \cite{alabau}).
 We  extend to our model the nice results achieved in \cite{alabau} for the case of nonlinearly damped Timoshenko system with thermoelasticity. The proof of our results relies on the approach in \cite{AB2, AB1}.
\end{abstract}

\blfootnote{MSC codes:
\subjclass{35B35, 35B40, 35L51, 93D20.}\\
Key words:\textit{\ Lower bounds, Optimality, Thermoelasticity, Timoshenko system, Strong asymptotic stability.}\newline
* Corresponding author: mahamoud@indiana.edu.}

\section{Introduction}

\label{sect1}
Mecanical structures such  as beams and plates are a central part of life today, their  vibration properties are extensively investigated by many researchers. These vibrations are  undesirable because of their damaging and destructing nature. To reduce  these harmful vibrations, several control mechanisms have been disigned. In order to do that, it is nutural to model and undrestand the corresponding equations of  these problems.\\
In this article we are concerned with the following nonlinearly damped
Timoshenko system in a one-dimensional bounded domain with thermoelasticity
where the heat flux is given by the Cattaneo's law:
\begin{equation}
\left\{
\begin{array}{l}
\rho _{1}\varphi _{tt}-k(\varphi _{x}+\psi )_{x}=0,\hspace{6.2cm} \mbox{
in}\ (0,1)\times \mathrm{I\hskip-2ptR}_{+}, \\
\rho _{2}\psi _{tt}-b\psi _{xx}+k(\varphi _{x}+\psi )+\delta \theta
_{x}+a(x)g(\psi _{t})=0,\hspace{2.2cm}\mbox{in} \ (0,1)\times \mathrm{I%
\hskip-2ptR}_{+}, \\
\rho _{3}\theta _{t}+q_{x}+\delta \psi _{xt}=0,\hspace{6.82cm}\mbox{in}\
(0,1)\times \mathrm{I\hskip-2ptR}_{+}, \\
\tau q_{t}+\beta q+\theta _{x}=0,\hspace{7.3cm}\mbox{in} \ (0,1)\times
\mathrm{I\hskip-2ptR}_{+}.%
\end{array}%
\right.  \label{1}
\end{equation}%
We associate with \eqref{1} the following Dirichlet boundary conditions
\begin{equation}
\varphi (0,t)=\varphi (1,t)=\psi (0,t)=\psi (1,t)=q(0,t)=q(1,t)=0,\hspace{%
0.65cm}\forall \mbox{ }t\geq 0.  \label{cb}
\end{equation}%
Moreover, the initial conditions for the system \eqref{1}  are given by :
\begin{equation}
\left\{
\begin{array}{l}
\varphi (x,0)=\varphi _{0}(x),\mbox{ }\varphi _{t}(x,0)=\varphi _{1}(x),%
\hspace{4.8cm}\forall \mbox{ }x\in \ (0,1), \\
\psi (x,0)=\psi _{0}(x),\mbox{ }\psi _{t}(x,0)=\psi _{1}(x),\hspace{4.8cm}%
\forall \mbox{ }x\in\  (0,1), \\
\theta (x,0)=\theta _{0}(x),\mbox{ }q(x,0)=q_{0}(x),\hspace{5.2cm}\forall %
\mbox{ }x\in \ (0,1),%
\end{array}%
\right.  \label{ci}
\end{equation}%
where $t\in (0,\infty )$ denotes the time variable and $x\in (0,1)$ is the
space variable, the function $\varphi $ is the displacement vector, $\psi $
is the rotation angle of the filament, the function $\theta $ is the
temperature difference, $q=q(x,t)\in \mathbb{R}$ is the heat flux, and $\rho
_{1}$, $\rho _{2}$, $\rho _{3}$, $b$, $k$, $\delta $ and $\beta $ are positive
constants.\newline

The Timoshenko model describes the vibration of a beam when the transverse
shear strain is significant. In 1920, Timoshenko \cite{45} introduced a
purely conserved hyperbolic system given by
\begin{equation}
\left\{
\begin{array}{l}
\rho _{1}\varphi _{tt}-k(\varphi _{x}+\psi )_{x}=0,\hspace{6.2cm}\mbox{%
in}\ (0,1)\times \mathrm{I\hskip-2ptR}_{+}, \\
\rho _{2}\psi _{tt}-b\psi _{xx}+k(\varphi _{x}+\psi )=0,\hspace{5.1cm}%
\mbox{in}\ (0,1)\times \mathrm{I\hskip-2ptR}_{+}.%
\end{array}%
\right.  \label{TIMOS}
\end{equation}%
The well understanding of this model was the goal of a great number of
researchers, thus, an important amount of research has been devoted to the
issue of the stabilization of the Timoshenko system by the use of diverse
types of dissipative mechanisms aiming to obtain a solution which decays
uniformly to the stable state as time goes to infinity. To achieve this goal
several upper energy estimates have been derived. For an overview purpose,
we shall mention some known results in this regard.
Kim and Renardy \cite{16}, Messaoudi and Mustafa \cite{19}, Raposo et al.
\cite{Timo13}, and others, showed that the presence of damping terms on both
equations \eqref{TIMOS} leads to uniform stability result regardless of the
values of the damping coefficients. The situation is much different, when
the damping term is only imposed on the rotation angle equation in the Timoshenko
system. In this case, the exponential stability holds if and only if the propagation velocities are equal. It is worth noting that the first result including the linear and nonlinear indirect damping cases and showing polynomial stability for different speeds of propagation was established in \cite{AB2}
giving thus optimal results in the nonlinear damping case
 (and getting as a particular case the exponential decay for the same speeds of propagation $(\frac{k}{\rho_1}=\frac{b}{\rho_2})$); see \cite{AB2, noeq1, noeq2, noeq, rivera} and the references therein.\\
Concerning stabilization via heat effect, Rivera and Racke \cite{MunozRiveraRackeMildy} investigated the following system
\[\left\lbrace
\begin{array}{l}
\rho _{1}\varphi _{tt}-\sigma (\varphi _{x},\psi )_{x}=0,\hspace{3.2cm}\mbox{ in }\,(0,L)\times \mathrm{I\hskip-2ptR}_{+}, \\
\rho _{2}\psi _{tt}-b\psi _{xx}+k(\varphi _{x}+\psi )+\gamma \theta _{x}=0,\hspace{0.8cm}\mbox{ in }\,(0,L)\times \mathrm{I\hskip-2ptR}_{+}
,\\
\rho _{3}\theta _{t}-k\theta _{xx}+\gamma \psi _{xt}=0,\hspace{3cm}\mbox{ in }(0,L)\times \mathrm{I\hskip-2ptR}_{+},
\end{array}\right.
\]
where $\varphi ,\psi ,\theta $ are functions of $(x,t)$ model the transverse
displacement of the beam, the rotation angle of the filament, and the
difference temperature respectively. Under appropriate conditions of $\sigma
,\rho _{i},b,k,\gamma ,$ they proved several exponential decay results for
the linearized system and non exponential stability result for the case of
different wave speeds.

Concerning Timoshenko systems of thermoelasticity with second sound,
Messaoudi \textit{et al.} \cite{MessaoudiPokojovy} studied
\[\left\lbrace
\begin{array}{l}
\rho _{1}\varphi _{tt}-\sigma (\varphi _{x},\psi )_{x}+\mu \varphi
_{t}=0,\hspace{2.6cm}\mbox{ in }\,(0,L)\times  \mathrm{I\hskip-2ptR}_{+} ,\\
\rho _{2}\psi _{tt}-b\psi _{xx}+k(\varphi _{x}+\psi )+\beta \theta
_{x}=0,\hspace{1.3cm}\mbox{ in }\, (0,L)\times  \mathrm{I\hskip-2ptR}_{+}, \\
\rho _{3}\theta _{t}+\gamma q_{x}+\delta \psi _{tx}=0,\hspace{3.7cm}\mbox{ in }\,
(0,L)\times  \mathrm{I\hskip-2ptR}_{+} ,\\
\tau _{0}q_{t}+q+\kappa \theta _{x}=0,\hspace{4.3cm}\mbox{ in }\, (0,L)\times  \mathrm{I\hskip-2ptR}_{+},
\end{array}\right.
\]
where $\varphi =\varphi (x,t)$ is the displacement vector, $\psi =\psi (x,t)$
is the rotation angle of the filament, $\theta =\theta (x,t)$ is the
temperature difference, $q=q(x,t)$ is the heat flux vector, $\rho _{1}$, $%
\rho _{2}$, $\rho _{3}$, $b$, $k$, $\gamma $, $\delta $, $\kappa $, $\mu $, $%
\tau _{0}$ are positive constants. The nonlinear function $\sigma $ is
assumed to be sufficiently smooth and satisfy
\[
\sigma _{\varphi _{x}}(0,0)=\sigma _{\psi }(0,0)=k
\]
and
\[
\sigma _{\varphi _{x}\varphi _{x}}(0,0)=\sigma _{\varphi _{x}\psi
}(0,0)=\sigma _{\psi \psi }=0.
\]
Several exponential decay results for both linear and nonlinear cases have
been established in the presence of the extra frictional damping $\mu
\varphi _{t}$.

Fern\'{a}ndez Sare and Racke \cite{18} considered
\begin{equation}
\left\lbrace
\begin{array}{l}
\rho _{1}\varphi _{tt}-k(\varphi _{x}+\psi )_{x}=0,\hspace{3.4cm}\mbox{ in }\,
(0,L)\times  \mathrm{I\hskip-2ptR}_{+}, \\
\rho _{2}\psi _{tt}-b\psi _{xx}+k(\varphi _{x}+\psi )+\beta \theta
_{x}=0,\hspace{1.33cm}\mbox{ in }\, (0,L)\times  \mathrm{I\hskip-2ptR}_{+}, \\
\rho _{3}\theta _{t}+\gamma q_{x}+\delta \psi _{tx}=0,\hspace{3.7cm}\mbox{ in }\,
(0,L)\times  \mathrm{I\hskip-2ptR}_{+}, \\
\tau _{0}q_{t}+q+\kappa \theta _{x}=0,\hspace{4.3cm}\mbox{ in }\,  (0,L)\times  \mathrm{I\hskip-2ptR}_{+},
\end{array}\right. \label{Timo6}
\end{equation}

and showed that, in the absence of the extra frictional damping ($\mu =0$),
the coupling via Cattaneo' s law causes loss of the exponential decay
usually obtained in the case of coupling via Fourier' s law \cite{MunozRiveraRackeMildy}. This
surprising property holds even for systems with history of the form
\begin{equation}
\left\lbrace
\begin{array}{l}
\rho _{1}\varphi _{tt}-k(\varphi _{x}+\psi )_{x}=0,\hspace{7.07cm}\mbox{ in }\,
(0,L)\times  \mathrm{I\hskip-2ptR}_{+}, \\
\rho _{2}\psi _{tt}-b\psi _{xx}+k(\varphi _{x}+\psi )+\int_{0}^{+\infty
}g(s)\psi _{xx}(.,t-s)ds+\beta \theta _{x}=0,\hspace{0.55cm}\mbox{ in }\,
(0,L)\times  \mathrm{I\hskip-2ptR}_{+}, \\
\rho _{3}\theta _{t}+\gamma q_{x}+\delta \psi _{tx}=0, \hspace{7.33cm}\mbox{ in }\,
(0,L)\times  \mathrm{I\hskip-2ptR}_{+},\\
\tau _{0}q_{t}+q+\kappa \theta _{x}=0,\hspace{7.92cm}\mbox{ in }\,
(0,L)\times  \mathrm{I\hskip-2ptR}_{+},
\end{array}\right. \label{Timo7}
\end{equation}
 Precisely, it has been shown that both systems (\ref{Timo6}) and (\ref{Timo7}) are no longer
exponentially stable even for equal-wave speeds $\left( \frac{k}{\rho _{1}}=%
\frac{b}{\rho _{2}}\right) .$ However, no other rate of decay has been
discussed.\\
Very recently, Santos et al. \cite{Santos} considered (\ref{Timo6}) and introduced a new
stability number
\[
\chi =\left( \tau _{0}-\frac{\rho _{1}}{k\rho _{3}}\right) \left( \rho _{2}-%
\frac{\rho _{1}b}{k}\right) -\frac{\rho _{1}\beta ^{2}\rho _{1}}{k\rho _{3}}
\]
and used the semigroup method to obtain exponential decay result for $\chi =0
$ and a polynomial decay for $\chi \neq 0.$\\
Later, in \cite{ali} the authors considered a vibrating nonlinear Timoshenko system with thermoelasticity with second sound.  Precisely, they looked into the following system
 \begin{equation}
\left\{
\begin{array}{l}
\rho _{1}\varphi _{tt}-k(\varphi _{x}+\psi )_{x}=0,\hspace{6.2cm}\textnormal{in }%
(0,1)\times \mathrm{I\hskip-2ptR}_{+}, \\
\rho _{2}\psi _{tt}-b \psi _{xx}+k(\varphi _{x}+\psi )
+\delta \theta _{x}+\alpha (t)h(\psi _{t})=0,\hspace{2.1cm} \textnormal{in }(0,1)\times
\mathrm{I\hskip-2ptR}_{+}, \\
\rho _{3}\theta _{t}+q_{x}+\delta \psi _{xt}=0,\hspace{6.73cm} \textnormal{in }(0,1)\times
\mathrm{I\hskip-2ptR}_{+}, \\
\tau q_{t}+\beta q+\theta _{x}=0,\hspace{7.2cm} \textnormal{in }(0,1)\times \mathrm{I\hskip%
-2ptR}_{+}, \\
\varphi _{x}(0,t)=\varphi _{x}(1,t)=\psi (0,t)=\psi (1,t)=q(0,t)=q(1,t)=0,%
\hspace{0.5cm}\forall\mbox{ }t\geq 0 ,\\
\varphi (x,0)=\varphi _{0}(x),\mbox{ }\varphi _{t}(x,0)=\varphi _{1}(x),\hspace{4.8cm}
\forall\mbox{ }x\in (0,1) ,\\
\psi (x,0)=\psi _{0}(x),\mbox{ }\psi _{t}(x,0)=\psi _{1}(x),\hspace{4.8cm}\forall\mbox{ }x\in
(0,1), \\
\theta (x,0)=\theta _{0}(x),\mbox{ }q(x,0)=q_{0}(x),\hspace{5.2cm}\forall\mbox{ }x\in (0,1),
\end{array}
\right.
\label{all}
\end{equation}
and they etablished an explicit and general decay result using a multiplier method for wide classe relaxating function without imposing the usual growth conditions on the frictional damping in both cases when $\chi=0$ and $\chi\neq 0.$

On the other hand, deriving the upper estimates is only a first step and it
needs to be completed by the obtention of the lower estimates. However, very
few is known on lower energy estimates and optimality results. Let us
mention the existing results in this regard. Haraux \cite{Haraux} examined
the case of a one-dimensional wave equation subjected to polynomial globally
distributed dampings, for some initial data in $W^{2,\infty }(\Omega )\times
W^{1,\infty }(\Omega ).$ Haraux proved that
\begin{equation*}
\limsup_{t\rightarrow \infty }(t^{\frac{3}{p-1}}E(t))>0,
\end{equation*}%
where $E(t)$ is the energy associated with the damped wave equation, and,
\begin{equation*}
\limsup_{t\rightarrow \infty }\left( t^{\frac{1}{p-1}}\Vert u_{t}\Vert
_{L^{\infty }(\Omega )}\right) >0,
\end{equation*}%
 where $g$ is a nondecreasing
$\mathcal{C}^1$ function which behaves essentially like $k|s|^r s$ with $k, r > 0$ and the damping term $g(x)$ grows as $x^{p}$ near the origin. Since that
time, this issue retains the attention of many other authors. We also refer
to \cite[Chapter1]{editor} for more details about the stabilization of
wave-like equations.\newline

More precisely, lower energy estimates have been previously studied in the
articles \cite{1}, \cite{2} for the scalar one-dimensional wave equations,
the scalar Petrowsky equations in two-dimensions and $(2\times 2)$ Timoshenko systems.
\newline
Let us also quote the article of Alabau \cite{alabau} for recent studies on strong
lower energy estimates of the strong solutions of nonlinearly damped
Timoshenko beams, Petrowsky equations, in two and three dimensions, and
wave-like equations, in a bounded one-dimensional domain or annulus domains
in two or three dimensions. Note nevertheless that considering the system (%
\ref{1}) makes our lower bound results more general from those considered so far in the
literature.\newline

The main objective of the present paper is to show how the energy $E$ (defined by (%
\ref{energy}) blow) associated with the nonlinearly damped Timoshenko system
of thermoelasticity with second sound (\ref{1}) satisfies the 
stability result. Once we have this stability result, one can use the expression of the energy $\mathcal{E}$ (defined by \ref{EE} blow)and  apply the
comparison principle which allows us to give the strong lower estimates for
the system (\ref{1})  

The rest of the article is organized as follows. We start in Section \ref{sect2} by
giving a brief introduction, then we introduce some notations and material
needed for our work. In Sections \ref{sect3} we state and prove the 
stabilization result for (\ref{1}). Then in Section \ref{sect4} we derive
the lower energy estimates for the Timoshenko system (\ref{1}). %
  Some exemples are given in the last section.

\section{Preliminaries}

\label{sect2}

We formulate the following assumptions that would be required for the
establishment of our results:
$(H_{0})$: we assume that  $a$ is a smooth function and satisfies $a(x) \geq 0$, $x \in  ]0,1[$ , $a > 0$ in a nonempty subset $]0,1[$ of $\omega$;
\begin{equation*}
\hspace{0.3cm}(H_{1})\left\{
\begin{array}{l}
\ g:\mathbb{R}\rightarrow \mathbb{R}\ \mbox{ is a nondecreasing }\ C^{0}-%
\mbox{function} \\
\mbox{
such   that for every }\ \epsilon \in (0,1),%
\mbox{ there exists positive
constants}\ c_{1},c_{2}, \\
\mbox{and an increasing odd function}\ g_{0}\in C^{1}(0,+\infty
),g_{0}(0)=0\ \mbox{ such that} \\
\newline
\displaystyle\left\{
\begin{array}{l}
g_{0}(|(s)|)\leq |g(s)|\leq g_{0}^{-1}(|(s)|),\hspace{2.7em}\mbox{for all}%
\hspace{1.2em}|s|\leq \epsilon , \\
c_{1}|s|\leq |g(s)|\leq c_{2}|s|,\hspace{6.2em}\mbox{ for all}\hspace{1.13em}%
|s|\geq \epsilon .%
\end{array}%
\right.%
\end{array}%
\right.
\end{equation*}%
In addition, we assume that, there exists $r_{0}>0$ such that $\Psi $ is a
strictly convex $\mathcal{C}^{1}-$function from $[0,r_{0}^{2}]$ on to $%
\mathbb{R}$, given by,
\begin{equation}
\Psi (x)=\sqrt{x}g(\sqrt{x}).  \label{psi}
\end{equation}

\begin{remarks} \mbox{ }
\begin{enumerate}
\item The function $\Psi$ defined above is the same function $H$ introduced in \cite{AB1}.
\item In \cite{alabau} Alabau assumed that $g$ is an odd, increasing function and
has a linear growth at infinity. In order to establish here the lower estimates,
the hypotheses in \cite{alabau} are only assumed for the  function $g_0$ and not for $g$.
\end{enumerate}
\end{remarks}

The energy associated with the system (\ref{1}) is defined by
\begin{equation}  \label{energy}
\displaystyle E(\varphi,\psi ,\theta ,q )(t):=\frac{1}{2}\int_{0}^{1}
(\rho_{1}\varphi_{t}^{2}+\rho_{2}\psi _{t}^{2}+b\psi
_{x}^{2}+k(\varphi_{x}+\psi )^{2}+\rho_{3}\theta^{2}+\tau q^{2})dx.
\end{equation}
Differentiating (\ref{energy}) in time, it is easy to see that
\begin{equation}  \label{energyde}
E^{\prime }(t)=-\beta
\int^{1}_{0}q^{2}dx-\int^{1}_{0}a(x)\psi_{t}g(\psi_{t}) dx \leq 0,
\end{equation}
this relationship has been obtained by multiplying, formally, the first fourth
equations of \eqref{1}, respectively, by $\varphi_{t}$, $\psi_{t}$, $\theta $
and $q$, and using the integration by parts with respect to $x$ over $(0,1),$
the boundary and initial conditions, and the hypotheses $(H_{0})$ and $%
(H_{1})$.

Now, we define the function space associated with the problem \eqref{1} by
\begin{equation*}
\mathcal{H}=H_{0}^{1}(0,1)\times L^{2}(0,1)\times H_{0}^{1}(0,1)\times
(L^{2}(0,1))^{3}.
\end{equation*}%
We rewrite  \eqref{1} as a first-order system. For that purpose, let $%
U=(\varphi ;\varphi _{t};\psi ;\psi _{t};\theta ;q)^{T}$ and \eqref{1}
becomes
\begin{equation}
\left\{
\begin{array}{c}
\displaystyle\frac{d}{dt}U(t)+(\mathcal{A}+B)U(t)=0,\hspace{1cm}t>0,\vspace{.2cm}\\
\vspace{%
0.2cm}U(0)=U_{0}=(\varphi _{0},\varphi _{1},\psi _{0},\psi _{1},\theta
_{0},q_{0})\in \mathcal{H},%
\end{array}%
\right.  \label{firstorder}
\end{equation}%
where $\mathcal{A}$ is an unbounded operator from $D(\mathcal{A})$ onto $%
\mathcal{H}$ defined by
\begin{equation}
\mathcal{A}\left(
\begin{array}{c}
\varphi \\
w \\
\psi \\
z \\
\theta \\
q%
\end{array}%
\right) =\left(
\begin{array}{c}
-w \\
-\frac{k}{\rho _{1}}\varphi _{xx}-\frac{k}{\rho _{1}}\psi _{x} \\
-z \\
-\frac{b}{\rho _{2}}\psi _{xx}+\frac{k}{\rho _{2}}(\varphi _{x}+\psi )+\frac{%
\delta }{\rho _{2}}\theta _{x} \\  \\
\frac{1}{\rho _{3}}q_{x}+\frac{\delta }{\rho _{3}}z_{x} \\  \\
\frac{\beta }{\tau }q+\frac{1}{\tau }\theta _{x}%
\end{array}%
\right) .  \label{a*fi}
\end{equation}%
Here,
\begin{equation*}
D(\mathcal{A})=((H^{2}(0,1)\cap H_{0}^{1}(0,1))\times
H_{0}^{1}(0,1))^{2}\times H^{1}(0,1)\times H_{0}^{1}(0,1).
\end{equation*}%
Clearly, $D(\mathcal{A})$ is dense in $\mathcal{H}$.\newline
Let $B$ be the damping nonlinear operator given by
\begin{equation*}
\displaystyle B\left(
\begin{array}{c}
\varphi \\
w \\
\psi \\
z \\
\theta \\
q%
\end{array}%
\right) =\left(
\begin{array}{c}
0 \\
0 \\
0 \\
a(x)g(w) \\
0 \\
0%
\end{array}%
\right) .
\end{equation*}%
Thanks to the theory of maximal nonlinear monotone operators (see \cite{wellp}%
), we have the following existence and uniqueness result (see \cite{ali} for
the proof). \newline

\begin{theorem}
\label{thp} Assume that $(H_{0})$ and $(H_{1})$ are satisfied. Then for all
initial data $U_{0}\in \mathcal{H}$, the system (\ref{1}) has a unique
solution $U\in \mathcal{C}([0,\infty );\mathcal{H})$, the operator $\mathcal{%
A}+B$ generates a continuous semigroup $(\mathcal{T}(t))_{t\geq 0}$ on $%
\mathcal{H}$. Moreover, for all initial data $U_{0}\in D(\mathcal{A})$, the
solution $U\in L^{\infty }([0,\infty );D(\mathcal{A}))\cap W^{1,\infty
}([0,\infty );\mathcal{H}).$
\end{theorem}

\begin{remark}
As we already mentioned in the indroduction, the exponential decay result  (\ref{Timo6}) depends on the stability number $\chi $  introduced in  \cite{Santos}. 	So, it is natural to wonder about the effects of the nonlinear dissipation mechanism $a(x)g(\psi_t)$ on the stability result of the system \eqref{1}. We recall that,
in  \cite{ali}, the authors considredred the same stability number $\chi$ and obtained a  general decay of the system \eqref{all} with a dissipation term of the form $\alpha(t)h(\psi_t)$ but no  optimality  result has been proved. \\
As a consequence, the following questions naturally arise:
\begin{itemize}
\item[$\bullet$] Is  our system  (\ref{1}) strongly stable? 
\item[$\bullet$] If we obtain a different equilibrum state $(E(t)\rightarrow \text{constant } \neq 0\  \text{as} \ t\rightarrow \infty )$, how can we characterize the decay rate of the energy?
\item[$\bullet$] Can we obtain  lower estimates for the new equilibrum state? 
\end{itemize}
These questions will be investigated in the next sections.
   \end{remark}
\section{ Stability for Timoshenko system}

\label{sect3}  In  this section, we focus on the stability result for the energy. For that purpose, we follow the following steps.\newline
 We consider frist the following  conservative Timoshenko  system:

 \begin{equation}\label{conservation}
\left\{
\begin{array}{l}
\rho _{1}\varphi_{tt}-k(\varphi _{x}+\psi )_{x}=0,\hspace{5.28cm}\mbox{in} \
(0,1)\times \mathrm{I\hskip-2ptR}_{+}, \\
\rho _{2}\psi_{tt}-b\psi _{xx}+k(\varphi _{x}+\psi )=0,\hspace{4.2cm}\mbox{in} \ (0,1)\times \mathrm{I\hskip%
-2ptR}_{+}.
\end{array}%
\right.
\end{equation}
Then, we assume the assumption  below  on the subset $\omega \subset 
]0,1[$,
\begin{equation*}
(HS)\left\{
\begin{array}{c}
\mbox{Let}\ (\varphi ,\psi)\
\mbox{ be a weak solution of \eqref{conservation}}  \\
\mbox{if}\ \psi _{t}\equiv 0\ \mbox{on}\ \omega \ %
\mbox{then}\ (\varphi ,\psi )\equiv (0,0).%
\end{array}%
\right.
\end{equation*}%
The assumption $(HS)$ is extracted from \cite{alabau} and we note that we proceed as in \cite{alabau} to extend the techniques there to our problem.\\
Now, we denote  by $\omega(U_0)$ the $\omega-$limit set of $U_0$ and we consider $Z_0 \in \omega(U_0)$ such that $Z(t)=\mathcal{T}(t)Z_0.$ 
Then we formulate the stability result for the energy of %
\eqref{1} in the following theorem.

\begin{theorem}
\label{TH1} Assume that the hypotheses $(H_0)$ and $(H_1)$ hold. We assume
in addition that $\omega$ satisfies $(HS)$. Then for all $%
U_0=(\varphi_0,\varphi_1,\psi_0,\psi_1,\theta_0,q_0)\in \mathcal{H}$, the
energy $E$ defined by \eqref{energy} corresponding to the solution of %
\eqref{1}, satisfies
\begin{equation}  \label{stab}
 {\lim_{t\rightarrow \infty}E(t,U)=E_{\infty},}
\end{equation}
where $E_{\infty}$ is  the energy of $Z \in \omega (U_0)$.\\
Moreover, under the same assumptions we prove that  the energy $\mathcal{E}(t)$  defined  in \eqref{EE} below, satisfies 
\begin{equation}
\mathcal{E}(t)\rightarrow 0, \hspace{0.3cm} \text{when} \hspace{0.2cm} t\rightarrow \infty.
\end{equation}
\end{theorem}


Before showing the proof of Theorem \ref{TH1}, we will state and prove two lemmas which
will be useful to this end. The first lemma below proves the decreasing of
the first order energy.

\begin{lemma}
\label{l1} Let $E_{\star}(t)$ be the energy defined as follows:
\begin{equation}
E_{\star}(t):=\frac{1}{2}\int_{0}^{1}
(\rho_{1}\varphi_{tt}^{2}+\rho_{2}\psi_{tt}^{2}+b\psi
_{tx}^{2}+k(\varphi_{tx}+\psi_t)^{2}+\rho_{3}\theta_t^{2}+\tau q_t^{2})dx.
\end{equation}
Then, $E_{\star}$ is a nonincreasing function. We shall call $E_{\star}(t)$
the first order energy associated with \eqref{1}.
\end{lemma}

\begin{proof}
We set $p=\varphi _{t},z=\psi _{t},u=\theta _{t},d=q_{t},$ then we have
\begin{equation}
\left\{
\begin{array}{l}
\rho _{1}p_{t}-k(\varphi _{x}+\psi )_{x}=0,\hspace{5.32cm}\mbox{in} \
(0,1)\times \mathrm{I\hskip-2ptR}_{+}, \\
\rho _{2}z_{t}-b\psi _{xx}+k(\varphi _{x}+\psi )+\delta \theta
_{x}+a(x)g(z)=0,\hspace{1.4cm}\mbox{in} \ (0,1)\times \mathrm{I\hskip%
-2ptR}_{+}, \\
\rho _{3}u+q_{x}+\delta z_{x}=0,\hspace{5.9cm}\mbox{in}\ (0,1)\times
\mathrm{I\hskip-2ptR}_{+}, \\
\tau d+\beta q+\theta _{x}=0,\hspace{6.19cm}\mbox{in}\ (0,1)\times
\mathrm{I\hskip-2ptR}_{+}.%
\end{array}%
\right.
\end{equation}

Differentiating the above equations with respect to time, we obtain
\begin{equation}  \label{12}
\left\{
\begin{array}{l}
\rho _{1}p_{tt}-k(\varphi _{x}+\psi )_{x}=0,\hspace{5.18cm}\mbox{in }\
(0,1)\times \mathrm{I\hskip-2ptR}_{+}, \\
\rho _{2}z_{tt}-b z _{xx}+k(p_x+z) +\delta u_{x}+a(x)g^{\prime }(z)z_t=0,%
\hspace{1cm} \mbox{in }\ (0,1)\times \mathrm{I\hskip-2ptR}_{+}, \\
\rho _{3}u_t+q_{x}+\delta z_x=0,\hspace{5.78cm} \mbox{in }\ (0,1)\times
\mathrm{I\hskip-2ptR}_{+}, \\
\tau d_t+\beta d+u_x=0,\hspace{6cm} \mbox{in }\ (0,1)\times \mathrm{I%
\hskip-2ptR}_{+}. \\
p(0,t)=p(1,t)=z(0,t)=z(1,t)=d(0,t)=d(1,t)=0,\hspace{0.3cm}\forall\mbox{ }%
t\geq 0.%
\end{array}
\right.
\end{equation}

\hspace{0.1cm} We remark that if we formally multiply the equations in %
\eqref{12}, respectively, by $p_{t}$, $z_{t}$, $u $ and $d$, integrate over $%
(0,1)$ and use the integration by parts with respect to $x$, the boundary
conditions, and the hypotheses $(H_{0})$ and $(H_{1})$, we obtain the
following inquality
\begin{equation}  \label{energyde1}
E_{\star}^{\prime }(t)=-\beta \int^{1}_{0}d^{2}dx-\int^{1}_{0}a(x)g^{\prime
}(z)z_t^2 dx \leq 0.
\end{equation}

Thus we deduce that $E_{\star}$ is nonincreasing, hence, we have
\begin{equation*}
E_{\star}(t)\leq E_{\star}(0), \hspace{1cm} \forall \ t\geq 0.
\end{equation*}
\end{proof}

We start by establishing the compactness of the orbit in the following lemma.

\begin{lemma}
\label{l2} For the initial data $U_0=(\varphi_0,\varphi_1,\psi_0,\psi_1,%
\theta_0,q_0)\in D(A)$, the orbit of $U_0$ given by $\gamma(U_0)=\cup_{t\geq
0}\mathcal{T}(t)U_0$ is relatively compact in $\mathcal{H}$.
\end{lemma}

\begin{proof}
\mbox{} Thanks to the first equation of \eqref{1}, we have
\begin{equation*}
\varphi _{xx}=-\frac{\rho _{1}}{k}(\varphi _{tt}+\psi _{x}),
\end{equation*}%
and we get
\begin{equation*}
\int_{0}^{1}\varphi _{xx}^{2}dx\leq \left( \frac{\rho _{1}}{k}\right)
^{2}\left( \int_{0}^{1}\varphi _{tt}^{2}dx+\int_{0}^{1}\psi
_{x}^{2}dx\right) .
\end{equation*}%
Using Lemma \ref{l1} which proves that $E_{\star }$ is bounded on $\mathbb{R}%
_{+}$, we deduce that the set $\{\varphi _{tt}(t,\cdot);t\geq 0\}$ is bounded in
$L^{2}(0,1)$. In addition to the fact that $E$ is bounded uniformly on $%
\mathbb{R}_{+}$, we deduce that the set $\{\psi _{x}(t,\cdot);t\geq 0\}$ is also
bounded in $L^{2}(0,1)$.\newline
Applying the Poincare's inequality and the Rellich-Kondrochov
theorem, we obtain that the set
\begin{equation*}
\{\varphi (t,\cdot);t\geq 0\}\ \mbox{ is relatively compact in }\ H_{0}^{1}(0,1).
\end{equation*}%
Thanks to \eqref{energyde1} the energy $E_{\star }$ is bounded in $\mathbb{R}%
_{+}$, then, the set $\{\varphi _{tx}(t,\cdot);t\geq 0\}$ is bounded in $%
L^{2}(0,1)$. Furthermore, we apply the Poincare's inequality
for $\varphi _{t}\in H_{0}^{1}(0,1)$,
\begin{equation*}
\Vert \varphi _{t}\Vert _{H_{0}^{1}(0,1)}\leq (1+c_{p})\Vert \varphi
_{tx}\Vert _{L^{2}(0,1)}.
\end{equation*}%
Hence, we easily obtain that the set $\{\varphi _{t}(t,\cdot);t\geq 0\}$ is
bounded in $H_{0}^{1}(0,1)$ which implies, using the Rellich theorem, that
the set
\begin{equation*}
\{\varphi _{t}(t,\cdot);t\geq 0\}\ \mbox{ is relatively compact in}\ L^{2}(0,1).
\end{equation*}%
From \eqref{1}, we have
\begin{equation*}
\theta _{x}=-\tau q_{t}-\beta q,
\end{equation*}%
and the sets $\{q(t,\cdot),t\geq 0\}$ and $\{q_{t}(t,\cdot),t\geq 0\}$ are bounded
in $L^{2}(0,1)$. Hence, we deduce that $\{\theta _{x}(t,\cdot),t\geq 0\}$ is
bounded in $L^{2}(0,1)$. Moreover, using the equation
\begin{equation*}
b\psi _{xx}=\rho _{2}\psi _{tt}+k(\varphi _{x}+\psi )+\delta \theta
_{x}+a(x)g(\psi _{t}),
\end{equation*}%
and the hypotheses $(H_{0})$ and $(H_{1})$ on $a$ and $g$, we obtain that $%
\{\psi _{xx}(t,\cdot),t\geq 0\}$ is bounded in $L^{2}(0,1)$, $\{\psi (t,\cdot),t\geq
0\}$ is bounded in $H^{2}(0,1)$ and again applying the Rellich-Kondrochov
theorem, we deduce that the set
\begin{equation*}
\{\psi (t,\cdot),t\geq 0\}\ \mbox{ is  relatively compact in }\ H_{0}^{1}(0,1).
\end{equation*}%
Since, we have the set $\{\psi _{t}(t,\cdot),t\geq 0\}$ is bounded in $%
H_{0}^{1}(0,1)$ we easily deduce from the Rellich theorem that
\begin{equation*}
\{\psi _{t}(t,\cdot),t\geq 0\}\ \mbox{is relatively compact in }\ L^{2}(0,1).
\end{equation*}%
Using the fact that $E_{\star }$ is bounded and
\begin{equation*}
q_{x}=-\delta \psi _{tx}-\rho _{3}\theta _{t},
\end{equation*}%
we infer that the $\{q_{x}(t,\cdot),t\geq 0\}$ is bounded in $L^{2}(0,1)$.
Applying the Poincare's inequality, than the set $%
\{q(t,\cdot),t\geq 0\}$ is bounded in $H_{0}^{1}(0,1)$, which infers that
\begin{equation*}
\{q(t,\cdot),t\geq 0\}\ \mbox{is relatively compact in }\ L^{2}(0,1).
\end{equation*}%
Using the fourth equation of \eqref{1}, we deduce that $\{\theta
_{x}(t,\cdot),t\geq 0\}$ is bounded in $L^{2}(0,1)$, as well, $\forall \ t\geq 0$%
, $\theta _{x}(t,\cdot)\in H_{0}^{1}(0,1)$, then $\{\theta _{x}(t,\cdot),t\geq 0\}$
is bounded in $H_{0}^{1}(0,1)$. Therefore we conclude that
\begin{equation*}
\{\theta (t,\cdot),t\geq 0\}\ \mbox{is relatively compact in }\ L^{2}(0,1).
\end{equation*}
\end{proof}

Now, we recall the definition of the $\omega-$limit that we borrow from \cite%
{editor}.

\begin{definition}
\label{wlimit} Let $(\mathcal{T}(t))_{t\geq 0}$ be a continuous semigroup on
a Banach $X$. We recall that the $\omega-$limit set of $z_0$, in $X$, is
defined by
\begin{equation*}
\omega(z_0)=\lbrace z \in X, \exists (t_n)_n \subset [0,\infty)\ \mbox{such that}%
\ t_n\rightarrow \infty, \ \mbox{as}\  n \rightarrow \infty, \ \mbox{and} \ z=\lim_{n\rightarrow \infty} \mathcal{T}(t_n)z_0
\rbrace.
\end{equation*}
\end{definition}

Now, we are ready to give the proof of Theorem \ref{TH1}.

\begin{proof}
[Proof of Theorem \ref{TH1}]\mbox{}\\
We aim to apply the Dafermos strong stabilization technique based on Lasalle
invariance principle (see Proposition 1.3.6 in \cite{editor}).\newline
For that purpose, let $U_{0}=(\varphi _{0},\varphi
_{1},\psi _{0},\psi _{1},\theta _{0},q_{0})\in D(\mathcal{A})$, and $%
U=(\varphi ,p,\psi ,z,\theta ,q)=\mathcal{T}(t)U_{0}.$ Then, we define the
Liapunov function $L$ for $(\mathcal{T}(t))_{t\geq 0}$ on $\mathcal{H}$ by
\begin{equation*}
\displaystyle L(U)=\frac{1}{2}\int_{0}^{1}(\rho _{1}p^{2}+\rho
_{2}z^{2}+b\psi _{x}^{2}+k(\varphi _{x}+\psi )^{2}+\rho _{3}\theta ^{2}+\tau
q^{2})dx,\hspace{0.3cm}\forall \ U\in \mathcal{H}.
\end{equation*}%
Now, let $\omega (U_{0})$ be the $\omega -$limit of $U_{0}$ (see
Definition \ref{wlimit}). Thanks to the Lasalle invariance principle,
we show that for each $W_{0}$ in $\omega (U_{0})$, the function $%
t\rightarrow L(\mathcal{T}(t)W_{0})$ is constant. In particular, let $%
Z_{0}\in \omega (U_{0})$ be given and we set $Z(t)=(w,r,z,p,u,\eta )(t)=%
\mathcal{T}(t)Z_{0}.$ 
Since $L(Z(\cdot))$ is constant, we deduce that $(w,z,u,\eta )$ is a solution of
a conservative system. 
Then,  the dissipation inequality will be equal to zero which yields
\begin{equation*}
-\beta \int_{0}^{1}\eta ^{2}dx-\int_{0}^{1}a(x)pg(p)dx=0\Rightarrow \eta
\equiv 0\ \mbox{and}\ a(x)g(p)=0,\forall \ x\in \ (0,1),\forall \ t\in \mathbb{%
R}_{+}.
\end{equation*}%
Hence, the conservative  system can be written as follows:
\begin{equation}
\left\{
\begin{array}{l}
\rho _{1}w_{tt}-k(w_{x}+z)_{x}=0,\hspace{5.97cm}\mbox{in} \ (0,1)\times
\mathrm{I\hskip-2ptR}_{+}, \\
\rho _{2}z_{tt}-bz_{xx}+k(w_{x}+z)=0,\hspace{5.1cm}\mbox{in}\ (0,1)\times
\mathrm{I\hskip-2ptR}_{+}, \\
\rho _{3}u_{t}+\delta p_{x}=0,\hspace{7.4cm}\mbox{in}\ (0,1)\times
\mathrm{I\hskip-2ptR}_{+}, \\
u_x=0, \hspace{8.8cm}\mbox{in}\ (0,1)\times
\mathrm{I\hskip-2ptR}_{+}, \\
z_{t}=0,\ \mbox{on}\ \{x\in \Omega ,a(x)\neq 0\}\supset \omega .%
\end{array}%
\right.   \label{newsys}
\end{equation}%
This yields,

\begin{equation}
\left\{
\begin{array}{l}
\rho _{1}w_{tt}-k(w_{x}+z)_{x}=0,\hspace{5.96cm}\mbox{in}\ (0,1)\times
\mathrm{I\hskip-2ptR}_{+}, \\
\rho _{2}z_{tt}-bz_{xx}+k(w_{x}+z)=0,\hspace{5.07cm}\mbox{in} \ (0,1)\times
\mathrm{I\hskip-2ptR}_{+}, \\
u_t=0, \hspace{8.8cm} 	\mbox{in} \ (0,1)\times \mathrm{I\hskip-2ptR}_{+},\\
u_x=0, \hspace{8.7cm} 	\mbox{in} \ (0,1)\times \mathrm{I\hskip-2ptR}_{+},\\
q=0, \hspace{9cm} 	\mbox{in} \ (0,1)\times \mathrm{I\hskip-2ptR}_{+},\\
z_{t}=0\ \{x\in \Omega ,a(x)\neq 0\}\supset \omega.\\
\end{array}%
\right.   \label{newsys0}
\end{equation}%
as well as, we can infer from \eqref{newsys0} that,

\begin{equation}
\left\{
\begin{array}{l}
\rho _{1}w_{tt}-k(w_{x}+z)_{x}=0,\hspace{5.96cm}\mbox{in}\ (0,1)\times
\mathrm{I\hskip-2ptR}_{+}, \\
\rho _{2}z_{tt}-bz_{xx}+k(w_{x}+z)=0,\hspace{5.1cm}\mbox{in} \ (0,1)\times
\mathrm{I\hskip-2ptR}_{+}, \\
u=c=\theta_0(0) \hspace{7.8cm} 	\mbox{in} \ (0,1)\times \mathrm{I\hskip-2ptR}_{+},\\
q=0, \hspace{9cm} 	\mbox{in} \ (0,1)\times \mathrm{I\hskip-2ptR}_{+}.\\
z_{t}=0\ \{x\in \Omega ,a(x)\neq 0\}\supset \omega .\\
\end{array}%
\right.   \label{newsys1}
\end{equation}%

Using the assumption $(HS)$, we have that $(w,z)= (0,0)$. This allows us to identify $Z(t)$ the element   of $ \omega(U_0)$  in the form  $Z(t)=(0,0,\theta_0(0),0).$
 Hence we conclude that,
\begin{equation*}
\lim_{t\rightarrow \infty} E(t,U_0)=E(Z)=E_{\infty}, \hspace{1cm} \forall\  U_0 \  \in D(\mathcal{A}).
\end{equation*}%
Indeed, since $D(\mathcal{A})$ is dense in $\mathcal{H}$, we obtain 
\begin{equation}
\lim_{t\rightarrow \infty} E(t,U)= E_{\infty}.
\end{equation}

Moreover,
since $\mathcal{E}$ is the energy of the difference between the solution $U \in \mathcal{H}$ and $Z=(0,0,\theta_0(0),0)  \in \omega(U_0),$ we obtain  
\begin{equation}\label{EE}
\mathcal{E}(t)=E(t,(\varphi,\psi,\theta-\theta_0(0),q))= \frac{1}{2} \int_0^1 \rho_1 {\varphi_t}^2+ \rho_2 {\psi_t}^2+b {\psi_x}^2+k(\varphi_x+\psi)^2+\rho_3 (\theta- \theta_0(0))^2 +\tau q^2 dx.
\end{equation} 
Thanks to the dissipation inequality \eqref{energyde} and  \eqref{psi},
we have
\begin{equation*}
\mathcal{E}^{\prime }(t,U)=-\beta \int_{0}^{1}q^{2}(t,x)dx-\int_{0}^{1}a(x)\psi_t ^{2}(t,x)%
\tilde{\Psi}(\psi_t ^{2}(t,x))dx.
\end{equation*}%
We assume that $ \mathcal{E}^{\prime }(t)=0$, $\forall \ t\geq 0$ and that  the hypothesis $(HS)$ holds, we obtain the following system 
\begin{equation}
\left\{
\begin{array}{l}
\rho _{1}w_{tt}-k(w_{x}+z)_{x}=0,\hspace{5.96cm}\mbox{in}\ (0,1)\times
\mathrm{I\hskip-2ptR}_{+}, \\
\rho _{2}z_{tt}-bz_{xx}+k(w_{x}+z)=0,\hspace{5cm}\mbox{in} \ (0,1)\times
\mathrm{I\hskip-2ptR}_{+}, \\
u-\theta_0(0)=0 \hspace{7.7cm} 	\mbox{in} \ (0,1)\times \mathrm{I\hskip-2ptR}_{+}\\
q=0, \hspace{8.9cm} 	\mbox{in} \ (0,1)\times \mathrm{I\hskip-2ptR}_{+},\\
u_t=0, \hspace{8.7cm} 	\mbox{in} \ (0,1)\times \mathrm{I\hskip-2ptR}_{+},\\
u_x=0, \hspace{8.6cm} 	\mbox{in} \ (0,1)\times \mathrm{I\hskip-2ptR}_{+},\\
z_{t}=0\ \{x\in \Omega ,a(x)\neq 0\}\supset \omega .\\
\end{array}%
\right.   \label{newsys11}
\end{equation}%
Then for this case the set $\omega(U_0-Z)=\lbrace (0,0,0,0)\rbrace.$  Applying  the Deformos'strong stabilisation technique as before, we obtain  that  
\begin{equation}
\lim_{t\rightarrow \infty}\mathcal{E}(t)=0 
\end{equation} 
\end{proof}

A straightforward consequence of the stabilisation result given by Theorem %
\ref{TH1} is stated in the following lemma.

\begin{lemma}
\label{lama1}{ For any $r_0>0$, there exists $T_0 > 0$ such that\newline
\begin{equation}\label{lamaaa}
\mathcal{E}(t)\leq \left(\frac{r_0^2}{\mathcal{\gamma}}\right)^2, \hspace{1cm} \forall
\ t\geq T_0.
\end{equation}}
\end{lemma}
\begin{proof}
Using the strong stability result given by Theorem \ref{TH1}, the energy $\mathcal{E}(t)$ converges to $0$ when $t$ tends to $\infty$. Then, energy $\mathcal{E}$ is bounded uniformly on $\mathbb{R}^{+}$.
In particular, we take the initial data $(\varphi_0,\varphi_1,\psi_0,\psi_1,\theta_0,q_0)$ such that $\mathcal{E}(0)\leq \left(\frac{r_0^2}{\gamma}\right)^2$, where $\gamma$ is defined later in \eqref{gamma}. Hence, we deduce \eqref{lamaaa}.
\end{proof}

\section{Lower energy estimates}

\label{sect4} 
 The aim of this section is to establish a lower bound of
the energy of the one-dimensional nonlinearly damped Timoshenko system with
thermoelasticity and also to prove that the method based on the comparison
principle, expressed through the energy of the solutions, can be extended to
our case.\newline
First, we define \textcolor{red}{(as in \cite{AB1})} the function $\Lambda $ as follows
\begin{equation}
\Lambda (x)=\frac{\Psi (x)}{x\Psi ^{\prime }(x)},  \label{gamma}
\end{equation}%
and we introduce the following assumption (which is the hypothesis (H2) in \cite{alabau})
\begin{equation*}
\hspace{0.3cm}(H_{2})\left\{
\begin{array}{l}
\exists r_{0}>0\ \mbox{such that the function }\ \Psi :[0,r_{0}^{2}]\mapsto
\mathbb{R}\ \mbox{defined by }(\ref{psi}) \\
\mbox{is strictly convex on }\ [0,r_{0}^{2}], \\
\mbox{and either}\ 0<\liminf_{x\rightarrow 0}\Lambda (x)\leq
\limsup_{x\rightarrow 0}\Lambda (x)<1, \\
\mbox{or there exists}\ \mu >0\ \mbox{such that} \\
0<\liminf_{x\rightarrow 0}\left( \frac{\Psi (\mu x)}{\mu x}\int_{x}^{z_{1}}%
\frac{1}{\Psi (y)}dy\right) ,\ \mbox{and}\ \limsup_{x\rightarrow 0}\Lambda
(x)<1, \\
\mbox{for some}\ z_{1}\in (0,z_{0}] \ \mbox{and for all} \ z_0>0.%
\end{array}%
\right.
\end{equation*}%
Then, we state in the sequel our main result.

\begin{theorem}
\label{TH3} Assume that $(H_0)$, $(H_1)$ and $(H_2)$ hold. Then for all non
vanishing smooth initial data, there exist $T_0 > 0$ and $T_1 > 0$ such that
the energy $\mathcal{E}$ of \eqref{1} satisfies the following lower estimate
\begin{equation}  \label{lower2}
\frac{1}{\mathcal{\gamma}^2 C_{\sigma}^2}\left( \Psi^{\prime -1}\left(\frac{1%
}{t-T_0}\right)\right)^2\leq \mathcal{E}(t), \hspace{1cm} \forall \ t\geq T_1+T_0.
\end{equation}
\end{theorem}
\begin{remark}
 {The result of Theorem \ref{TH3} holds true without any assumption on the wave speeds corresponding to the first two equations in \eqref{1}, see e.g \cite{1, editor, alabau}}. 
\end{remark}
The proof of Theorem \ref{TH3} relies on the next proposition together with
Lemma \ref{lama55} which is proved in \cite[Lemma 2.4]{2} and based on the approach of \cite{AB1}. We reproduce here the details for the sake of completeness.
\begin{proposition}
\label{TH2} Let
\begin{equation*}
U_0=(\varphi_0,\varphi_1,\psi_0,\psi_1,\theta_0,q_0)\in D(A).
\end{equation*}
We assume that the hypotheses of Theorem \ref{thp} hold and $%
\lim_{t\rightarrow \infty} \mathcal{E}(t) = 0$. Moreover, we assume that
\begin{equation}  \label{r}
\tilde{\Psi}(x)= \frac{\Psi(x)}{x},\hspace{0.5cm} \tilde{\Psi}(0)=0, \hspace{%
0.2cm} \forall \ x>0,
\end{equation}
where $\Psi$ is a nondecreasing function on $[0, r_0^2]$ for $r_0 > 0$ sufficiently small.%
\newline
Then there exists $T_0 \geq 0$, depending on $E_1(0)$ such that, defining $%
\mathcal{K}$ by
\begin{equation}
\mathcal{K}(\chi)=\int_{\chi}^{\mathcal{\gamma} \sqrt{\mathcal{E}(T_0)}}\frac{1}{%
\Psi(y)}dy,\hspace{1cm} \chi \in (0, \mathcal{\gamma}\sqrt{\mathcal{E}(T_0)}),
\end{equation}
the energy $\mathcal{E}$ satisfies the following lower estimate

\begin{equation}  \label{13}
\frac{1}{\mathcal{\gamma}^2} \left(\mathcal{K}^{-1}(\sigma (t-T_0)
\right)^2\leq \mathcal{E}(t).
\end{equation}
Here, $\sigma$ is a positive constant given by $\sigma= \frac{
\alpha_a}{\rho_2}+ \frac{\beta r_0}{\tau C_1},$ where $\alpha_{a}$ and $%
\mathcal{\gamma}$ are defined, respectively, by \eqref{alpha1} and %
\eqref{gt} below. \newline
Moreover, if $\lim_{\chi\rightarrow 0^{+}} \mathcal{K}(\chi)=\infty$, then
\begin{equation*}
\lim_{t\rightarrow \infty} \mathcal{K}^{-1}(\sigma (t-T_0))=0.
\end{equation*}
\end{proposition}

\begin{proof}
We assume that the initial data $U_{0}\in D(A)$. Then, thanks to the
smoothness of the solution, we have
\begin{equation*}
2\int_{0}^{x}\psi _{t}(t,z)\psi _{tx}(t,z)dz=\psi _{t}^{2}(t,x)-\psi
_{t}^{2}(t,0).
\end{equation*}%
Using to the Dirichlet boundary conditions \eqref{cb} at $x=0$, we have
\begin{equation*}
\psi _{t}^{2}(t,x)=2\int_{0}^{x}\psi _{t}(t,z)\psi _{tx}(t,z)dz.
\end{equation*}%
Applying the Cauchy-Schwarz inequality, we have
\begin{align*}
\psi _{t}^{2}(t,x)\leq 2\sqrt{\left( \int_{0}^{1}\psi _{t}^{2}(t,\cdot)dx\right)
}\sqrt{\left( \int_{0}^{1}\psi _{tx}^{2}(t,\cdot)dx\right) }
\\
\leq \frac{4}{\rho _{2}}\sqrt{\mathcal{E}(t)}%
\sqrt{E_{\star }(0)}, \hspace{0.6cm}%
\forall \ x\in (0,1), \ \forall \ t\geq 0.
\end{align*}%
Using \eqref{energyde1} and the fact that $E_{\star}(t)\leq E_{\star}(0)$ , we deduce that

\begin{equation*}
\psi _{t}^{2}(t,x)\leq \mathcal{\gamma }\sqrt{\mathcal{E}(t)}, \hspace{0.6cm}\forall \ t\geq 0,%
\hspace{0.5cm} \forall \ x\in (0,1),
\end{equation*}%
where $\mathcal{\gamma }$ is given by
\begin{equation}
\mathcal{\gamma }=\frac{4}{\rho _{2}}\sqrt{E_{\star }(0)}.  \label{gt}
\end{equation}%
Thanks to Theorem \ref{TH1}, we have $\psi _{t}\in W^{1,\infty }(0,\infty
,L^{2}(0,1)).$ From \eqref{gt}  and the above regularity of $\psi _{t}$, we
have
\begin{equation}
\Vert \psi _{t}^{2}(t,\cdot)\Vert _{L^{\infty }(0,1)}\leq \mathcal{\gamma }\sqrt{%
E(t)},\hspace{1cm}\forall \ t\geq 0.
\end{equation}%

Thanks to the dissipation inequality \eqref{energyde} and using \eqref{psi},
we have
\begin{equation*}
E^{\prime }(t)=-\beta \int_{0}^{1}q^{2}(t,x)dx-\int_{0}^{1}a(x)\psi_t ^{2}(t,x)%
\tilde{\Psi}(\psi_t ^{2}(t,x))dx.
\end{equation*}%
 On the other hand, from  the experession of the energy $\mathcal{E}$ we have the following relation between $\mathcal{E}^{\prime}$ and $E^{\prime}$
\begin{equation*}
\mathcal{E}^{\prime}(t)= E^{\prime}(t,U)-\rho_3 \theta_0(0) \frac{d}{dt}\left(\int_0^1 \theta(t,x) dx \right).
\end{equation*}
Using \eqref{1} and the boundary conditions, we have
$\frac{d}{dt}\int^{1}_{0} \theta (x,t)dx=0,$   $\mathcal{E}^{\prime}(t)=E^{\prime}(t,U).$\\

Moreover, using the Dafermos strong stabilization result, that is $%
\lim_{t\rightarrow \infty } \mathcal{E}(t)=0$, we deduce that there exists $T_{0}\geq 0$
such that $\psi _{t}^{2}$ has values in which $\tilde{\Psi}$ is increasing.%
\newline
Hence, we have
\begin{equation*}
\tilde{\Psi}(|\psi _{t}^{2}(t,\cdot)|)\leq \tilde{\Psi}(\mathcal{\gamma }\sqrt{%
\mathcal{E}(t)}),\hspace{1cm}\forall \ t\geq T_{0},\ \forall \ x\in (0,1).
\end{equation*}%
Using the last inequality we obtain
\begin{equation}
\int_{0}^{1}a(x)\psi_t ^{2}(t,x)\tilde{\Psi}(\psi_t ^{2}(t,x))dx\leq \frac{%
2 \alpha _{a}}{\rho_2\mathcal{\gamma }}\sqrt{\mathcal{E}(t)}\Psi (\mathcal{\gamma }%
\sqrt{\mathcal{E}(t)}),\hspace{1cm}\forall \ t\geq T_{0},
\end{equation}%
where
\begin{equation}
\alpha _{a}=\Vert a\Vert _{L^{\infty }(0,1)}.  \label{alpha1}
\end{equation}%
Moreover, using Lemma \ref{lama1}, we obtain
\begin{equation*}
\mathcal{E}(t)^{\frac{1}{4}}\leq \left( \frac{r_{0}}{\mathcal{\gamma }^{1/2}}\right) ,%
\hspace{1cm}\forall \ t\geq T_{0}.
\end{equation*}

\begin{itemize}
\item[\protect\underline{First case.}] Let $g_{0}$ be a linear function on $%
[0,\epsilon ]$, the hypothesis ($H_{1}$) becomes 
\begin{equation*}
c_{1}^{\star}|s|\leq |g(s)|\leq c_{2}^{\star }|s|,\hspace{1cm}\mbox{for all}
\ s\in \mathbb{R}.
\end{equation*}

In particular, for $s=\mathcal{\gamma}^{\frac{1}{2}} (\mathcal{E}(t))^{\frac{1}{4}},$
and, note that $g:\mathbb{R}\rightarrow \mathbb{R}$, then we have
\begin{equation}  \label{case1}
\mathcal{\gamma}^{\frac{1}{2}} (\mathcal{E}(t))^{\frac{1}{4}}\leq \frac{1}{c_1^{\star}}
g(\mathcal{\gamma}^{\frac{1}{2}} (\mathcal{E}(t))^{\frac{1}{4}}) \leq \frac{1}{%
c_1^{\star}}\Psi(\mathcal{\gamma}\sqrt{\mathcal{E}(t)}), \hspace{0.5cm} \forall \
t\geq T_0.
\end{equation}
\item[\protect\underline{Second case.}] \ \newline
Let $g_{0}$ be a nonlinear function on $[0,\epsilon ]$. We assume that $\max
(r_0,g_{0}(r_0))<\epsilon $.\newline
Let $\epsilon _{1}=\min (r_0,g_{0}(r_0))$, we deduce from the hypothesis $%
(H_ {1}) $ that
\begin{equation*}
\frac{g_{0}(\epsilon _{1})}{\epsilon }|s|\leq \frac{g_{0}(|s|)}{|s|}|s|\leq
|g(s)|\leq \frac{g_{0}^{-1}(|s|)}{|s|}|s|\leq \frac{g^{-1}_{0}(\epsilon )}{%
\epsilon _{1}}|s|,
\end{equation*}%
for all $s$ satisfying $\displaystyle \epsilon _{1}\leq |s|\leq \epsilon $.%
\newline
Using the fact that $|\mathcal{\gamma}^{\frac{1}{2}}(\mathcal{E}(t))^{\frac{1}{4}%
}|\leq r_0, $ we infer that

\begin{equation}  \label{case2}
\frac{g_{0}(\epsilon _{1})}{\epsilon } \mathcal{\gamma}^{\frac{1}{2}}(\mathcal{E}(t))^{%
\frac{1}{4}}\leq g(\mathcal{\gamma}^{\frac{1}{2}}(\mathcal{E}(t))^{\frac{1}{4}})\leq
\Psi(\mathcal{\gamma}\sqrt{\mathcal{ E}(t)}), \hspace{0.5cm} \forall \ t\geq T_0.
\end{equation}
\end{itemize}

Now, thanks to \eqref{case1} and \eqref{case2}, we deduce that for the two
cases we obtain the following estimate
\begin{align*}
\beta \int_0^1 q^2(t,x)dx\leq \frac{2\beta }{\tau} \mathcal{E}(t) \leq \frac{2\beta r_0
\sqrt{\mathcal{E}(t)}}{\tau C_1 \mathcal{\gamma}}\Psi(\mathcal{\gamma}\sqrt{\mathcal{E}(t)}),
\hspace{1cm} \forall \ t\geq T_0,
\end{align*}
where $C_1$ is a positive constant.\newline
Hence, there exists $T_0\geq0 $ such that the following inequality holds
\begin{equation}
-\mathcal{E}^{\prime }(t)\leq \frac{2}{\mathcal{\gamma}}\left(\frac{\alpha_a
}{\rho_2} + \frac{\beta r_0}{\tau C_1}\right) \sqrt{\mathcal{E}(t)}\Psi(\mathcal{\gamma%
}\sqrt{\mathcal{E}(t)}), \hspace{0.3cm} \forall \ t\geq T_0.
\end{equation}
Thus we deduce that
\begin{equation*}
\mathcal{K}(\mathcal{\gamma} \sqrt{\mathcal{E}(t)})\leq \left(\frac{ \alpha_a%
}{\rho_2 }+ \frac{\beta r_0 }{\tau C_1}\right)(t-T_0).
\end{equation*}
Since $\mathcal{K}$ is a nonincreasing function, this completes the proof of %
\eqref{13}.
\end{proof}
Now, we will use the following key comparison with the result borrowed from
Lemma 2.4 in \cite{2}.
\begin{lemma}
\label{lama55} Let $\Psi $ be a given strictly convex set of $C^1$ function from
$[0,r_0^2]$  to $\mathbb{R}$ such that $H(0)=H^{\prime}(0)=0$, where $%
r_0>0$ and sufficiently small. Let us define $\mathbf{\Lambda}$ on $(0,
r^2_0]$ by
\begin{equation}
\mathbf{\Lambda}(x)=\frac{H(x)}{x H^{\prime }(x)}.
\end{equation}
Let $z$ be the solution of the following ordinary differential equation:
\begin{equation}  \label{edo}
z^{\prime}(t) + \kappa H(z(t)) = 0 ,\hspace{0.5cm} z(0) = z_0, \forall \ t
\geq 0 ,
\end{equation}
where $z_0 > 0$ and $\kappa > 0$ are given. Then $z(t)$ is well defined for all
$t \geq 0$, and it decays to $0$, as $t \rightarrow \infty$. Assume, in addition,
that $(H_2)$ holds. Then there exists $T_1 > 0$ such that for all $R > 0,$
there exists a constant $C > 0$ such that
\begin{equation}  \label{lemmaeq}
z(t) \geq C(H^{\prime})^{-1}\left(\frac{R}{t}\right),\hspace{0.5cm}\forall \
t \geq T_1.
\end{equation}
\end{lemma}
\begin{proof}[Proof of Theorem \ref{TH3}]\mbox{}\\
Let $z(t)$ be the solution of the ordinary differential equation \eqref{edo}, where we assume that $z_0= \mathcal{\gamma} \sqrt{ \mathcal{E}(T_0)}$, $H=\Psi$ and $%
\kappa=\sigma$.\newline
Hence, we have
\begin{equation*}
\mathcal{K}(z(t))=\left(\frac{ \alpha_a}{\rho_2 }+ \frac{\beta r_0
}{\tau C_1}\right) t, \hspace{0.5cm}\forall t\geq 0.
\end{equation*}
We set $\widehat{z}(t)=z(t-T_0)$, then we have

\begin{equation*}
\widehat{z}(t)=\mathcal{K}^{-1}\left( \left(\frac{ \alpha_a}{\rho_2
}+ \frac{\beta r_0 }{\tau C_1}\right) (t-T_0)\right), \hspace{0.5cm} \forall
\ t\geq T_0.
\end{equation*}
Thanks to \eqref{13}, we have
\begin{equation}  \label{low1}
\frac{\widehat{z}(t)^2}{\mathcal{\gamma}^2} \leq \mathcal{E}(t), \hspace{0.5cm}
\forall \ t\geq T_0.
\end{equation}
We apply Lemma \ref{lama55} to $\Psi=H$ for $R=1$, then, we obtain the existence of  a constant $C_{\mathcal{\gamma}}$ depending on $\mathcal{\gamma}$ and a
positive constant $T_1$, such that
\begin{equation}  \label{low2}
(\Psi^{\prime})^{-1}\left( \frac{1}{t}\right)\leq C_{\mathcal{\gamma}} z(t),
\hspace{0.5cm} \forall \ t\geq T_1.
\end{equation}
By using \eqref{low1} and \eqref{low2}, we deduce that
\begin{equation*}
(\Psi^{\prime})^{-1}\left( \frac{1}{t-T_0}\right)\leq C_{\mathcal{\gamma}}
\widehat{z}(t), \hspace{0.5cm} \forall \ t\geq T_0+ T_1.
\end{equation*}
Hence, we have \eqref{lemmaeq}.
\end{proof}
We conjecture that driving the lower estimates leads to optimal energy decay rates in general. However, the proof of such a result is open.

\section{Examples}

\label{sect5} Throughout this section, we will first introduce some examples
which allow us to illustrate the main advantages of our results. Let  $c^{\prime }$  be a positive constant  explicitly given here and it only depends on the constant $\sigma $.\newline

\begin{itemize}
\item[\protect\underline{Example 1.}] Let $g(s)=s^p$, $\forall s\in
(0,r_0^2] $ for $p>1$.\newline
We have $\Psi(s)=s^{\frac{p+1}{2}},$ $\Psi$ is strictly convex, for $s \in
[0,r_0^2],$ and $\Psi^{\prime }(s)=\frac{p+1}{2} s^{\frac{p-1}{2}}$, then
\begin{equation*}
\tilde{\Psi}(s)=\frac{\Psi(s)}{s}=s^{p-1}, \ \mbox{for} \ p>1, \ \forall s
\in ]0,r_0^2].
\end{equation*}
Thus, $\tilde{\Psi}$ is nondecreasing on $]0,r_0^2].$ \newline
Since $\Lambda(x)=\frac{2}{p+1}<1,$ this proves that $g$ satisfies the first
assumption of $(H_2)$.\newline
By applying \eqref{lower2} of Theorem \ref{TH3}, we obtain
\begin{equation}  \label{exp1}
\mathcal{E}(t)\geq c^{\prime}\   t^{ \frac{-4}{p-1}}.
\end{equation}

\item[\protect\underline{Example 2.}] Let $g(s)= \frac{1}{s} exp(-(ln
(s))^2) $ , for all $s\in (0,r_0^2]$. This yields
\begin{equation*}
\Psi(s)=exp\left(\frac{-1}{4}(ln( s))^2\right),
\end{equation*}
and
\begin{equation*}
\Psi^{\prime }(s)=-\frac{ln (s)}{2s} exp\left(-\frac{1}{4}(ln (s))^2\right),
\end{equation*}
\begin{equation*}
\tilde{\Psi}(s)=\frac{\Psi(s)}{s}= \frac{1}{s} exp\left(\frac{-1}{4}(ln(
s))^2\right) , s \in (0,r_0^2].
\end{equation*}
Thus, $\tilde{\Psi}$ is nondecreasing on $]0,r_0^2].$\newline
In addition, $\Lambda(s)=\frac{-2}{ln (s)},$ thus, $\lim_{s\rightarrow
0}\Lambda (s)=0<1,$ and we get also for any $\mu>1$,
\begin{equation*}
\liminf_{s\rightarrow 0}\frac{\Psi(\mu s)}{\mu s}\int_s^{z_1}\frac{1}{\Psi(y)%
}dy>0.
\end{equation*}
It is easy to see that $\Psi^{\prime }(t)$ is equivalent to $D(t),$ as $t$
goes to $\infty$, where $D(t)= exp(-\frac{1}{4}(ln(t))^2).$\newline
So, we have $D^{-1}(t)=exp(- 2(ln(t)^{\frac{1}{2}}));$ here we apply the
result of the Theorem \ref{TH3} and we obtain the following inequality
\begin{equation*}
\mathcal{E}(t)\geq c^{\prime }\hspace{0.1cm} exp(- 4(ln(t)^{\frac{1}{2}})).
\end{equation*}
\end{itemize}
 {By these examples we obtain  explicit lower bounds  which characterize the decay rate of the energy $E(t)$, associated with the solution of \eqref{1},  to the  correponding non-zero equilibrium state energy $E_{\infty}$. }

\end{document}